\pdfoutput=1
\RequirePackage{ifpdf}
\ifpdf 
\documentclass[pdftex]{sigma}
\else
\documentclass{sigma}
\fi

\numberwithin{equation}{section}

\newtheorem{Theorem}{Theorem}[section]

\newtheorem{Lemma}[Theorem]{Lemma}
\newtheorem{Proposition}[Theorem]{Proposition}
\newtheorem{Fact}[Theorem]{Fact}
 { \theoremstyle{definition}
\newtheorem{Definition}[Theorem]{Definition}

\newtheorem{Remark}[Theorem]{Remark}
\newtheorem{Setting}[Theorem]{Setting}
\newtheorem{Notation}[Theorem]{Notation}
}

\begin{document}

\newcommand{\arXivNumber}{2005.03308}

\renewcommand{\PaperNumber}{042}

\FirstPageHeading

\ShortArticleName{Linear Independence of Generalized Poincar\'e Series for Anti-de Sitter 3-Manifolds}

\ArticleName{Linear Independence of Generalized Poincar\'e Series\\ for Anti-de Sitter 3-Manifolds}

\Author{Kazuki KANNAKA}

\AuthorNameForHeading{K.~Kannaka}

\Address{RIKEN iTHEMS, Wako, Saitama 351-0198, Japan}
\Email{\href{mailto:kazuki.kannaka@riken.jp}{kazuki.kannaka@riken.jp}}

\ArticleDates{Received May 13, 2020, in final form April 13, 2021; Published online April 23, 2021}

\Abstract{Let $\Gamma$ be a discrete group acting properly discontinuously and isometrically on~the three-dimensional anti-de Sitter space $\mathrm{AdS}^{3}$, and $\square$ the Laplacian which is a second-order hyperbolic differential operator. We~study linear independence of a family of gene\-ra\-li\-zed Poincar\'e series introduced by Kassel--Kobayashi [\textit{Adv. Math.} \textbf{287} (2016), 123--236, arXiv:1209.4075], which are defined by~the $\Gamma$-average of certain eigenfunctions on $\mathrm{AdS}^{3}$. We~prove that the multi\-pli\-ci\-ties of~$L^{2}$-eigenvalues of the hyperbolic Laplacian $\square$ on $\Gamma\backslash\mathrm{AdS}^{3}$ are unbounded when $\Gamma$ is finitely generated. Moreover, we prove that the multiplicities of \textit{stable $L^{2}$-eigenvalues} for compact anti-de Sitter $3$-manifolds are unbounded.}


\Classification{58J50; 53C50; 22E40}

\section{Introduction}
A pseudo-Riemannian manifold is a smooth manifold $M$ equipped with
a smooth non-degenerate symmetric bilinear tensor $g$
 of signature $(p,q)$ on $M$.
It is called Riemannian if $q=0$, and Lorentzian if $q=1$.
As in the Riemannian case,
the Laplacian $\square_{M}:=\operatorname{div}_{M} \circ \operatorname{grad}_{M}$
is defined as a second-order differential operator on $M$. We~note that it is a
hyperbolic differential operator if $M$ is Lorentzian. We~write $L^2(M)$ for the Hilbert space of square-integrable functions on $M$
with respect to the Radon measure induced by the pseudo-Riemannian structure.
For $\lambda \in \mathbb{C}$, we denote by
\begin{gather*}
L^2_{\lambda}(M) :=
\big\{f \in L^2(M) \mid \square_{M}f = \lambda f \text{ in the weak sense}\big\}.
\end{gather*}
The set of $L^2$-eigenvalues
$\mathrm{Spec}_{d}(\square_{M}) := \big\{ \lambda \in \mathbb{C}
\mid L^2_{\lambda}(M)\neq0\big\}$
is called the \textit{discrete spectrum} of~$\square_{M}$.

Our interest is the multiplicities of $L^2$-eigenvalues $\lambda$ of $\square_{M}$,
denoted by
\begin{gather*}
\mathcal{N}_{M}(\lambda):= \operatorname{dim}_{\mathbb{C}}L^2_{\lambda}(M)
\in \mathbb{N}\cup\{\infty\}.
\end{gather*}
In the Riemannian case, the Laplacian is an elliptic differential operator
and the distribution of~its discrete spectrum has been investigated extensively,
such as the Weyl law for compact Riemannian manifolds.
However, it is not the case for non-Riemannian manifolds.
Kobayashi~\cite{Kob16}, and later Fox--Strichartz~\cite{FoSt18},
investigated the distribution of the discrete spectrum of
the Laplacian $\square_{M}$ of some pseudo-Riemannian manifolds, i.e., when
$M$ is the flat pseudo-Riemannian manifold $\mathbb{R}^{p,q}/\mathbb{Z}^{p+q}$
and is the Lorentzian manifold $S^1\times S^{q}$, respectively.

Let us recall some basic notions.
A \textit{discontinuous group} for a homogeneous manifold \mbox{$X\!=\!G/H$}
is a discrete subgroup $\Gamma$ of $G$ acting properly discontinuously
and freely on $X$ (Kobayashi \cite[Definition~1.3]{Ko01}). In~this case, the quotient space $X_{\Gamma}:=\Gamma\backslash X$ carries a
$C^{\infty}$-manifold structure such that
the quotient map $p_{\Gamma}\colon X\rightarrow X_{\Gamma}$ is a covering of
$C^{\infty}$ class, hence $X_{\Gamma}$ has a $(G,X)$-structure induced by $p_{\Gamma}$.
If we drop the assumption of freeness, $X_{\Gamma}$ is not always a manifold but
carries a nice structure called an orbifold or $V$-manifold.
Proper discontinuity is a more
serious assumption which assures $X_{\Gamma}$ to be Hausdorff in the quotient topology. We~remark that the action of a discrete subgroup $\Gamma$ on $X$
may fail to be properly discontinuous
when $H$ is noncompact. In~order to overcome this difficulty,
Kobayashi~\cite{Kob96} and Benoist~\cite{Ben96} established the properness criterion
for reductive $G$ generalizing the original criterion by Kobayashi~\cite{Kob89}.
Whereas discontinuous groups for the de Sitter space
$\mathrm{dS}^{n}:=\mathrm{SO}_{0}(n,1)/\mathrm{SO}_{0}(n-1,1)$
are always finite groups
(the Calabi--Markus phenomenon, see~\cite{CaMa62, Kob89}),
there are a rich family of discontinuous groups for
the anti-de Sitter space, see, e.g.,~\cite{Go85, Kob98, Sa00}. We~treat, in this article, the three-dimensional anti-de Sitter space
$\mathrm{AdS}^{3}:=\mathrm{SO}_{0}(2,2)/(\{\pm1\}\times\mathrm{SO}_{0}(2,1))$.

For $m\in \mathbb{N}$, we set
\begin{gather*}
\lambda_{m} := 4m(m-1).
\end{gather*}
We prove:
\begin{Theorem}
\label{mthm1}
For any finitely generated discontinuous group $\Gamma$ for
$\mathrm{AdS}^{3}$,
\begin{gather*}\lim_{m \to \infty}\mathcal{N}_{\Gamma\backslash\operatorname{AdS}^{3}}
(\lambda_{m}) = \infty.
\end{gather*}
\end{Theorem}

\begin{Remark}\quad
\begin{enumerate} \itemsep=0pt
\item[(1)]
A discontinuous group $\Gamma$ for $\mathrm{AdS}^{3}$
is called standard~\cite[Definition~1.4]{KaKob16} if
it is contained in a reductive subgroup of $\mathrm{SO}_{0}(2,2)$
which acts properly on $\mathrm{AdS}^{3}$ such as $\mathrm{SU}(1,1)$.
When~$\Gamma$ is torsion-free and standard,
Kassel--Kobayashi~\cite{KaKob19,KaKob20PJA} established the theory of spectral decomposition
of $L^{2}(\Gamma\backslash\mathrm{AdS}^{3})$ into eigenfunctions of the (hyperbolic) Laplacian.
Moreover, a stronger result than Theorem~\ref{mthm1} holds in this case:
$\mathcal{N}_{\Gamma\backslash\mathrm{AdS}^{3}}(\lambda_{m}) = \infty$
for sufficiently large $m \in \mathbb{N}$ (Kassel--Kobayashi~\cite{AnaKaKob}).
On the other hand, a full spectral decomposition is not known.
The construction of $L^{2}$-eigenfunctions by generalized Poincar\'e series still works for
the non-standard case, showing that $\lambda_{m}$ is an $L^{2}$-eigenvalue on
$\Gamma\backslash\mathrm{AdS}^{3}$ for sufficiently large $m\in\mathbb{N}$~\cite{KaKob16}.
Theorem~\ref{mthm1} is also applicable to non-standard $\Gamma$,
for example, in the case where $\Gamma$ is Zariski dense in $\mathrm{SO}(2,2)$.

\item[(2)]
The assumption that $\Gamma$ is finitely generated could be relaxed. In~fact, the exponential growth condition (see (\ref{C})) for
$\Gamma$-orbits is essential in the proof of Theorem~\ref{mthm1},
and there exist infinitely generated discontinuous groups $\Gamma$ satisfying (\ref{C})
and the conclusion of~Theorem~\ref{mthm1} holds for such $\Gamma$
(see Theorem~\ref{unbound} which is proved without finitely generated assumption).

\item[(3)]
An analogous statement to Theorem~\ref{mthm1} also holds when
$\Gamma\backslash\mathrm{AdS}^{3}$ is an orbifold.
See~Section~\ref{preliminary-Poincare} for the argument when
we drop the assumption that the $\Gamma$-action is free.
\end{enumerate}
\end{Remark}

Now we consider a small deformation of a discrete subgroup.
The study of \textit{stability} for~pro\-pe\-rness
was intiated by Kobayashi~\cite{Kob98} and Kobayashi--Nasrin~\cite{KoNa06}
and has been developed
by~Kassel~\cite{Ka12} and others.
Moreover, Kassel--Kobayashi~\cite{KaKob16} proved
the existence of infinite \textit{stable $L^2$-eigenvalues}
under any small deformation of discontinuous groups. In~this article, we~also consider
the multiplicities of stable $L^2$-eigenvalues (Definition~\ref{stab-mult})
and prove that they are unbounded.

To be precise, let $X_{n}$ be the $n$-fold covering of $X_{1}:=\mathrm{AdS}^{3}$
for $1\leq n \leq \infty$,
and $G_{n}$ the Lie group of its isometries.
Every compact anti-de Sitter $3$-manifold $M$ is of the form
$M\cong\Gamma\backslash X_{n}$ for some \textit{finite} $n$,
where $\Gamma (\subset G_{n})$ is a discontinuous group for $X_{n}$
by Kulkarni--Raymond \cite[Theorem~7.2]{KuRa85} and Klingler~\cite{Kl96}. We~take $n$ to be the smallest integer of this property.

Let $\mathrm{Hom}(\Gamma,G_{n})$ be the set of
group homomorphisms with compact-open topology,
and $\mathcal{U}_{\Gamma}$ the set of neighborhoods $W$ in
$\mathrm{Hom}(\Gamma,G_{n})$
of the natural inclusion $\Gamma \subset G_{n}$
such that for any~$\varphi \in W$, the map $\varphi$ is injective and
$\varphi(\Gamma)$ acts properly discontinuously on $X_{n}$.
One knows $\mathcal{U}_{\Gamma}\neq\varnothing$~\cite{Kl96,Kob98}.
By definition, $\lambda$ is a stable $L^2$-eigenvalue if
$\min_{\varphi \in W} \mathcal{N}_{\varphi(\Gamma)\backslash X_{n}}(\lambda)\neq0$
for some $W \in \mathcal{U}_{\Gamma}$.
Moreover, for any $\lambda \in \mathbb{C}$ and
any inclusion $W'\subset W$ in $\mathcal{U}_{\Gamma}$,
we have an obvious inequality
\begin{gather*}
\min_{\varphi \in W'}
\mathcal{N}_{\varphi(\Gamma)\backslash X_{n}}(\lambda)
\geq
\min_{\varphi \in W}
\mathcal{N}_{\varphi(\Gamma)\backslash X_{n}}(\lambda).
\end{gather*}

\begin{Definition}
\label{stab-mult}
For a compact anti-de Sitter $3$-manifold $M$, we say that
\begin{gather*}
\widetilde{\mathcal{N}}_{M}(\lambda):=
\sup_{W\in\mathcal{U}_{\Gamma}}\min_{\varphi \in W}
\mathcal{N}_{\varphi(\Gamma)\backslash X_{n}}(\lambda)
\end{gather*}
is the multiplicity of a stable $L^2$-eigenvalue $\lambda$.
\end{Definition}

There exist infinitely many $m\in \mathbb{N}$
such that $\widetilde{\mathcal{N}}_{M}(\lambda_{m})\geq1$,
namely $\lambda_{m}$ is a stable $L^2$-eigen\-value
for sufficiently large $m$ \cite[Corollary~9.10]{KaKob16}.
However, to the best knowledge of the author,
it is not known whether
$\widetilde{\mathcal{N}}_{M}(\lambda_{m})$
is finite. We~prove:
\begin{Theorem}\label{mthm2}
For any compact anti-de Sitter $3$-manifold $M$,
\begin{gather*}
\lim_{m \to \infty}
\widetilde{\mathcal{N}}_{M}(\lambda_{m})=\infty.
\end{gather*}
\end{Theorem}

The organization of this article is as follows.
A key step to our proof is to find a family of $L^2$-eigenfunctions
of $\square_{\mathrm{AdS}^{3}}$ with eigenvalue $\lambda_{m}$ on $\mathrm{AdS}^{3}$
for which the corresponding
``generalized Poincar\'e series'' are linearly independent, see Proposition~\ref{mainthm}. In~Section~\ref{Preliminary-AdS},
we recall some facts about
$L^2$-eigenfunctions of $\square_{\mathrm{AdS}^{3}}$
and their generalized Poincar\'e series which were
introduced in~\cite{KaKob16} as the $\Gamma$-average of these eigenfunctions. We~then give a uniform estimate of the ``pseudo-distance''
between the origin and the second closest point of each $\Gamma$-orbit
(see Section~\ref{inj-radii}). In~Section~\ref{proof-linear-indep}, we complete a proof of Proposition~\ref{mainthm}. In~Section~\ref{proof-thm2-section}, we prove a generalization of~Theorem~\ref{mthm2}
to the case of convex cocompact groups (Definition~\ref{convex_cocompact}).

\section{Preliminaries about the anti-de Sitter space}
\label{Preliminary-AdS}

In this section, we collect some preliminary results about $\mathrm{AdS}^{3}$. We~refer to \cite[Section~9]{KaKob16} where they illustrate their general theory
for reductive symmetric spaces $X=G/H$ in details in the special setting where
$X=\mathrm{AdS}^{3}$. See also~\cite{Kan19}.

Let $Q$ be a quadratic form on $\mathbb{R}^{4}$ defined by
$Q(x) = x_1^2 + x_2^2-x_3^2-x_4^2$ for $x= (x_1,x_2,x_3,x_4)$ and
we set
\begin{gather*}
\mathbb{H}^{2,1}:=\big\{x=(x_1,x_2,x_3,x_4) \in \mathbb{R}^4 \mid Q(x) = 1\big\}
\cong \mathrm{SO}_{0}(2,2)/\mathrm{SO}_{0}(2,1).
\end{gather*}
The tangent space $T_{x}(\mathbb{H}^{2,1})$ at $x\in\mathbb{H}^{2,1}$
is isomorphic to the orthogonal complement $(\mathbb{R} x)^{\bot}$ with respect to $Q$.
Then $-Q|_{(\mathbb{R} x)^{\bot}}$ is
a quadratic form of signature $(2,1)$ on
$T_{x}(\mathbb{H}^{2,1}) \cong (\mathbb{R} x)^{\bot}$ and
thus $-Q$ induces a Lorentzian structure on $\mathbb{H}^{2,1}$
with constant sectional curvature $-1$.
The~$3$-di\-men\-sional anti-de Sitter space
\begin{gather*}
\mathrm{AdS}^{3}:=\mathbb{H}^{2,1}/\{\pm 1\}
\cong \mathrm{SO}_{0}(2,2)/(\{\pm1\}\times\mathrm{SO}_{0}(2,1)),
\end{gather*}
inherits a Lorentzian structure through
the double covering $\pi\colon\mathbb{H}^{2,1}\rightarrow\mathrm{AdS}^{3}$.

\subsection{Some coordinates and ``pseudo-balls''}
In this subsection, we work with coordinates on $\mathbb{H}^{2,1}$ and
consider ``pseudo-balls'' in $\mathrm{AdS}^{3}$. We~identify
$\mathbb{H}^{2,1}$ with $\mathrm{SL}(2,\mathbb{R})$ using the isomorphism
\begin{align}
\begin{array}{@{}ccc}
\mathbb{H}^{2,1} & \xrightarrow{\cong} & \mathrm{SL}(2,\mathbb{R}), \\
x=(x_1,x_2,x_3,x_4) & \longmapsto & \begin{pmatrix}
x_1 + x_4 & -x_2 + x_3 \\
x_2 + x_3 & x_1 -x_4
\end{pmatrix}\!.\label{isom}
\end{array}
\end{align}
For $t\geq0$ and $\theta\in\mathbb{R}$, we use the notations
\begin{gather}
\label{k-a}
k(\theta) = \begin{pmatrix}
\cos\theta & -\sin\theta \\
\sin\theta & \cos\theta
\end{pmatrix}\!,\qquad
a(t) = \begin{pmatrix}
{\rm e}^t & 0 \\
0 & {\rm e}^{-t}
\end{pmatrix}\!.
\end{gather}
We embed $\mathbb{H}^{2,1}$ into $\mathbb{C}^{2}$ by
\begin{gather}
\label{embed}
x\mapsto (z_1,z_2) = \big(x_1 + \sqrt{-1}x_2, x_3 + \sqrt{-1}x_4\big).
\end{gather}
We note that $z_{1}\neq0$ if $x\in\mathbb{H}^{2,1}$.
Via the identification (\ref{isom}), we have
\begin{gather}
\label{polar}
(z_1,z_2) = \big((\cosh t){\rm e}^{\sqrt{-1}(\theta_1+\theta_2)},
(\sinh t){\rm e}^{\sqrt{-1}(\theta_1-\theta_2)}\big),
\end{gather}
if $x = k(\theta_1)a(t)k(\theta_2) \in \mathrm{SL}(2,\mathbb{R})$
(a ``polar coordinate''). In~particular, we have
\begin{gather*}
\cosh 2t = x_1^2 + x_2^2 + x_3^2 + x_4^2.
\end{gather*}
Next, we consider pseudo-balls on $\mathrm{AdS}^3$, as a special case of
Kassel--Kobayashi~\cite{KaKob16} for
reductive symmetric spaces.
\begin{Definition}
\label{def-pball}
For $x = (x_1,x_2,x_3,x_4)\in\mathbb{H}^{2,1}$, $\|x\|\in\mathbb{R}_{\geq0}$
is defined by
\begin{gather*}
\cosh\|x\|:=x_1^2 + x_2^2 + x_3^2 + x_4^2\qquad (= \cosh(2t)).
\end{gather*}
This function is invariant under $x\mapsto -x$, hence defines
a function on $\mathrm{AdS}^{3}$,
to be also denoted by $\|\cdot\|$ (a ``pseudo-distance'' from the origin).
The compact set
\begin{gather*}
B(R) := \big\{y\in\mathrm{AdS}^{3} \mid \|y\| \leq R\big\}
\end{gather*}
is called a pseudo-ball of radius $R$.
\end{Definition}

\subsection[Square-integrable eigenfunctions of the Laplacian on the anti-de Sitter space] {Square-integrable eigenfunctions of the Laplacian \\on the anti-de Sitter space}\label{preliminary-spherical}

In this subsection, we consider square-integrable eigenfunctions of
$\square_{\mathrm{AdS}^{3}}$ with eigenvalues $\lambda_{m}=4m(m-1)$. We~recall from \cite[Section~9]{KaKob16} the following decomposition of the open subset $\{Q > 0\}$ of
the flat pseudo-Riemannian manifold $\mathbb{R}^{2,2}=\big(\mathbb{R}^{4},Q({\rm d}x)\big)$:
\begin{align*}
\begin{array}{@{}ccc}
\{Q > 0\} & \xrightarrow{\cong} & \mathbb{R}_{>0} \times \mathbb{H}^{2,1}, \\[1ex]
x & \longmapsto & \big(\sqrt{Q(x)}, x/\sqrt{Q(x)}\big).
\end{array}
\end{align*}
Let $r=\sqrt{Q(x)}$. Then one has, see \cite[p.~215]{KaKob16},
\begin{align}
\label{decomposition-of-laplacian}
-r^2\square_{\mathbb{R}^{2,2}} = -\left(r\frac{\partial}{\partial r}\right)^2 - 2r\frac{\partial}{\partial r} + \square_{\mathbb{H}^{2,1}}.
\end{align}
Let $m$ be a positive integer and $k$ be a non-negative integer. In~the coordinates (\ref{embed}),
the homogeneous function $z_{1}^{-(k+2m)}z_{2}^{k}$ of degree $-2m$
 is harmonic with respect to $\square_{\mathbb{R}^{2,2}}$,
hence its restriction to the submanifold $\mathbb{H}^{2,1}$ is
an eigenfuction of $\square_{\mathbb{H}^{2,1}}$ with eigenvalue $\lambda_{m}=4m(m-1)$
by~the formula (\ref{decomposition-of-laplacian}).
Moreover, it is square-integrable with respect to the measure
$\sinh(2t){\rm d}\theta_{1}{\rm d}t{\rm d}\theta_{2}$ in~the polar coordinate (\ref{polar})
induced from the Lorentzian metric on $\mathbb{H}^{2,1}$,
as in the $k=0$ case \cite[Section~9]{KaKob16}.
This $L^{2}$-eigenfunction is invariant under $(z_{1},z_{2})\mapsto(-z_{1},-z_{2})$,
hence defines a~real analytic $L^{2}$-eigenfunction on $\mathrm{AdS}^{3}$
with eigenvalue $\lambda_{m}$, to be denoted by $\psi_{m,k}$.
The discrete spectrum
$\mathrm{Spec}_{d}(\square_{\mathrm{AdS}^{3}})$ coincides with
$\{\lambda_{m} \mid m \in \mathbb{N}\}$
and $L^{2}_{\lambda_{m}}\big(\mathrm{AdS}^{3}\big)$ is generated by
$\psi_{m,0}$ and its complex conjugate $\overline{\psi_{m,0}}$
as a representation of $\mathrm{SO}_{0}(2,2)$ (see \cite[Claim~9.12]{KaKob16}).
By~(\ref{polar}), we~have
\begin{align}
\label{polarFJ}
\psi_{m,k}(\pi(x))
= {\rm e}^{-2\sqrt{-1}(m\theta_1 + (m+k)\theta_2)}\tanh^{k}t\cosh^{-2m}t
\end{align}
for $x=k(\theta_1)\,a(t)\,k(\theta_2)\in\mathbb{H}^{2,1}$. We~refer to $\psi_{m,k}$ as a spherical function of type $(-m,m+k)$
in~accordance with the action of
$\mathrm{SO}(2)\times\mathrm{SO}(2)$.

\subsection{Convergence of generalized Poincar\'e series}\label{preliminary-Poincare}
In this subsection, we explain the fact about the discrete spectrum
of locally symmetric spaces by Kassel--Kobayashi~\cite{KaKob16}
in our $\mathrm{AdS}^{3}$ setting. We~use the following notation.
\begin{Notation}\quad
\begin{itemize}\itemsep=0pt
\item
Let $\grave{\ }G=\mathrm{PSL}(2,\mathbb{R})=
\mathrm{SL}(2,\mathbb{R})/\{\pm1\}$ and
$G=\grave{\ }G\times\grave{\ }G$.
\item
Let $\grave{\ }K=\mathrm{PSO}(2)=\mathrm{SO}(2)/\{\pm1\}$ and
$K = \grave{\ }K\times\grave{\ }K$.
\item
Let $E$ and $\grave{\ }E$ be respectively the identity elements of $G$ and $\grave{\ }G$.
\end{itemize}
\end{Notation}

\begin{Remark}
The double covering $\mathrm{SO}_{0}(2,2)\rightarrow G$ induces
an isomorphism $\mathrm{AdS}^{3}\cong G/\mathrm{diag} \grave{\ }G$ $(\cong \grave{\ }G)$.
From now on, we consider only discontinuous groups $\Gamma$ for $\mathrm{AdS}^{3}$
which are discrete subgroups of $G$.
This is enough for our purpose.
\end{Remark}

In order to study
$\operatorname{Spec}_{d}\big(\square_{\Gamma\backslash\mathrm{AdS}^3}\big)$,
Kassel--Kobayashi~\cite{KaKob16} considered the convergence and non-vanishing of
generalized Poincar\'e series
\begin{gather}
\label{Poincare-series}
\varphi^{\Gamma}(\Gamma x):=\sum_{\gamma \in \Gamma} \varphi\big(\gamma^{-1}x\big)
\end{gather}
for $K$-finite square-integrable eigenfunctions $\varphi$ of $\square_{\mathrm{AdS}^{3}}$.
For this, they used an analytic estimate of $\varphi$ and
a geometric estimate of the number of $\Gamma$-orbits
\begin{gather}
\label{orbit-count}
N_{\Gamma}(x,R) := \#\{\gamma \in \Gamma \mid \gamma x \in B(R)\}
\end{gather}
in the pseudo-ball $B(R)$ for $R>0$.
Since the $\Gamma$-action is properly discontinuous and $B(R)$ is compact,
we have $N_{\Gamma}(x,R)<\infty$.

The convergence of generalized Poincar\'e series is proved by~\cite{KaKob16} as follows.
For $g \in G$ and a~function $f$ on $\mathrm{AdS}^{3}$,
$\ell_{g}^{*}f$ is defined by $\ell_{g}^{*}f(x)=f\big(g^{-1}x\big)$.
\begin{Fact}[Kassel--Kobayashi~\cite{KaKob16}]\label{Poincare}
Let $\Gamma \subset G$ be a discontinuous group for $\mathrm{AdS}^{3}$
satisfying the exponential growth condition
\begin{gather}\label{C}
\exists A,a > 0,\qquad
\forall x\in \mathrm{AdS}^{3},\qquad
\forall R>0, \qquad
N_{\Gamma}(x,R) < A{\rm e}^{aR}.
\end{gather}
Then, for any $K$-finite eigenfunction $\varphi$ of $\square_{\mathrm{AdS}^{3}}$ with eigenvalue
$\lambda_{m}$
and any $g\in G$,
if $m > a$, then
$(\ell_{g}^{*}\varphi)^{\Gamma}$ $($see $\eqref{Poincare-series})$ is continuous and square-integrable on
$\Gamma\backslash\mathrm{AdS}^{3}$ and
an eigenfunction of~$\square_{\Gamma\backslash\mathrm{AdS}^{3}}$ with eigenvalue
$\lambda_{m}$.
\end{Fact}

\begin{Remark}\label{remPoincare}\quad
\begin{enumerate}\itemsep=0pt
\item[(1)]
Fact~\ref{Poincare} does not assert the non-vanishing of the series
$(\ell_{g}^{*}\varphi)^{\Gamma}$ which is more involved.
Kassel--Kobayashi~\cite{KaKob16} proved that
there exists $g\in G$ such that
$(\ell^{*}_{g}\psi_{m,0})^{\Gamma}\neq 0$
for sufficiently large $m \in \mathbb{N}$.
\item[(2)]
By \cite[Lemma~4.6.4]{KaKob16}, if a discontinuous group $\Gamma$ is sharp
in the sense of \cite[Definition~4.2]{KaKob16},
then $\Gamma$ satisfies the exponential growth condition (\ref{C}).
Moreover, Kassel~\cite{Ka09} and Gu\'eriataud--Kassel~\cite{GuKa17}
proved that finitely generated discontinuous groups for $\mathrm{AdS}^{3}$
are always sharp (see Fact~\ref{proper} below).
\item[(3)]
There exist discontinuous groups
which do not satisfy the exponential growth condition~(\ref{C}).
Indeed, for any increasing function $f\colon\mathbb{R}\rightarrow \mathbb{R}_{>0}$
and any $x\in\mathrm{AdS}^{3}$,
we~con\-st\-ructed
a discontinuous group $\Gamma_{f,x}$ for $\mathrm{AdS}^{3}$
satisfying $N_{\Gamma_{f,x}}(x,R) > f(R)$ for sufficiently large $R>0$ in~\cite{Kan19}.
\end{enumerate}
\end{Remark}

The conclusion of Fact~\ref{Poincare} still holds if
we drop the assumption that
$\Gamma$ acts freely on $X=\mathrm{AdS}^{3}$. In~this case, the quotient space $X_{\Gamma}=\Gamma\backslash X$ is an orbifold.
To formulate more precisely in the orbifold case, we observe that
the quotient space $X_{\Gamma}$ is Hausdorff, and carries a natural Radon measure
(see, e.g., \cite[Chapter~VII, Section~2, No.~2, Proposition~4]{Bourbaki-integral}).
A continuous function $g$ on $X_{\Gamma}$ is \textit{smooth} if
the pull-back $p_{\Gamma}^{*}g$ is a smooth function on $X$,
where $p_{\Gamma}\colon X\rightarrow X_{\Gamma}$ is the natural quotient
map. We~write $C_{c}^{\infty}(X_{\Gamma})$
for the set of smooth functions on $X_{\Gamma}$ with compact support.
For $g\in C_{c}^{\infty}(X_{\Gamma})$,
we define $\square_{X_{\Gamma}}g \in C_{c}^{\infty}(X_{\Gamma})$ by identifying it with
the $\Gamma$-invariant function $\square_{X}(p_{\Gamma}^{*}g)$.
For $\lambda\in \mathbb{C}$, we define
\begin{gather*}
L^{2}_{\lambda}(X_{\Gamma}) :=
\big\{f\in L^{2}(X_{\Gamma})\mid \forall g\in C_{c}^{\infty}(X_{\Gamma}),
\langle f,\square_{X_{\Gamma}} g\rangle_{X_{\Gamma}}=\lambda\langle f,g\rangle_{X_{\Gamma}}\big\}.
\end{gather*}
The discrete spectrum $\mathrm{Spec}_{d}(\square_{X_{\Gamma}})$ and
its multiplicity $\mathcal{N}_{X_{\Gamma}}$ are defined similarly
to the case where $\Gamma$ acts also freely.

\subsection{``Injectivity radii'' of anti-de Sitter 3-manifolds}\label{inj-radii}
Let $\Gamma$ be a discontinuous group for $\mathrm{AdS}^{3}$. In~this subsection, we give a uniform estimate of the pseudo-distance
between the origin and the second closest point of each $\Gamma$-orbit.

We recall that $\Gamma(\subset \grave{\ }G\times\grave{\ }G)$
acts isometrically on $\mathrm{AdS}^{3}(\cong\grave{\ }G)$ by
$(\gamma_{1},\gamma_{2})x=\gamma_{1} x\gamma_{2}^{-1}$
for~$(\gamma_{1},\gamma_{2})\in\Gamma$ and $x\in\grave{\ }G$. We~set
\begin{align}
\label{varepsilon}
\varepsilon_{\Gamma} := \inf_{(\gamma_1,\gamma_2)\in
\Gamma\setminus\{E\}}
\frac{1}{3}\big|\|\gamma_1\|-\|\gamma_2\|\big|.
\end{align}
By the inequality (see, e.g., \cite[Lemma~5.5]{Kan19})
\begin{gather*}
\|(g_1,g_2)x\|\geq \big | \|g_1\|-\|g_2\| \big | - \|x\|\qquad
\text{for}\quad (g_1,g_2) \in G \quad\text{and}\quad x \in \mathrm{AdS}^{3},
\end{gather*}
we get:
\begin{Lemma}\label{inj-rad}
If $\varepsilon_{\Gamma}>0$, then
$\gamma B(\varepsilon_{\Gamma}) \cap B(\varepsilon_{\Gamma})
= \varnothing$
for all $\gamma\in \Gamma\setminus\{E\}$.
\end{Lemma}

\begin{Proposition}\label{strong}
Let $\Gamma$ be a discrete subgroup of $G$ acting properly discontinuously
on $\mathrm{AdS}^{3}$.
Then there exists $g \in G$
satisfying $\varepsilon_{g^{-1}\Gamma g} > 0$.
\end{Proposition}
\begin{Remark}
One sees in the proof below that
the set of such $g$ is dense in $G$.
\end{Remark}
Proposition~\ref{strong} follows obviously
from the proper discontinuity of the $\Gamma$-action and the following lemma:
\begin{Lemma}\label{weak}
For any countable subset $\Gamma$ of
$G$,
there exists
$g\in G$
such that $\|\gamma_1\|\neq\|\gamma_2\|$
for all $\gamma=(\gamma_1,\gamma_2)\in
g^{-1}\Gamma g \setminus \{E\}$.
\end{Lemma}

\begin{proof}[Proof of Lemma~\ref{weak}]
For $\gamma \in \Gamma$,
the map $f_{\gamma}\colon G \rightarrow G$ defined by $g\mapsto g^{-1}\gamma g$
is real analytic.
For the analytic subset
$F = \{(g_1,g_2) \in G
\mid \|g_1\| = \|g_2\| \}$ of $G$, we claim that
the set $f_{\gamma}^{-1}(F)$ is a proper subset of $G$ if $\gamma\neq E$.
For this, we may assume $\gamma_1 \neq \grave{\ }E$ without loss of generality.
Then there exists
$g_1 \in \grave{\ }G$ satisfying $\|g_1^{-1}\gamma_1g_1\| \neq \|
\gamma_1\|$ as one can find $g_{1}$ depending on the three cases where
$\gamma_{1}$ is hyperbolic, parabolic, or elliptic.
Hence $(g_1,\grave{\ }E)\notin f_{\gamma}^{-1}(F)$ if $\|\gamma_{1}\|=\|\gamma_{2}\|$,
and $E\notin f_{\gamma}^{-1}(F)$ if not. Thus $f_{\gamma}^{-1}(F)$ is a proper subset of $G$.

Therefore the analytic set
$f_{\gamma}^{-1}(F)$ has no interior point, and thus so does the countable union
$\bigcup_{\gamma \in \Gamma \setminus \{E\}}
 f_{\gamma}^{-1}(F)$ by the Baire category theorem
(see, e.g., \cite[Theorem~2.2]{functional-analysis}).
Hence there exists an element
$g$ of $G\setminus\bigcup_{\gamma \in \Gamma \setminus \{E\}}
 f_{\gamma}^{-1}(F)$ and we have
$\|\gamma_1\|\neq\|\gamma_2\|$
for all $\gamma=(\gamma_1,\gamma_2)\in
g^{-1}\Gamma g \setminus \{E\}$.
\end{proof}

\section{Proof of Theorem~\ref{mthm1}}\label{proof-linear-indep}
In this section, we prove Theorem~\ref{mthm1}.

More generally, without finitely generated assumption of $\Gamma$,
we study linear independence of~the generalized Poincar\'e series
of the spherical functions $\psi_{m,k}$ of type $(-m,m+k)$
defined in~Section~\ref{preliminary-spherical}.
By choosing $k=3^{j}$ ($j=0,1,2,\ldots$),
we prove:
\begin{Theorem}
\label{unbound}
If $\Gamma$ is a discontinuous group for $\mathrm{AdS}^{3}$ satisfying
the exponential growth condition~\eqref{C}, then
\begin{gather*}
\lim_{m\to\infty}\mathcal{N}_{\Gamma\backslash\mathrm{AdS}^{3}}(\lambda_{m})=\infty.
\end{gather*}
\end{Theorem}
Theorem~\ref{mthm1} is a direct consequence of
Theorem~\ref{unbound} by Remark~\ref{remPoincare}(2).

\begin{Proposition}
\label{mainthm}
Let $\Gamma$ be a discrete subgroup of $G$ acting properly discontinuously
on $\mathrm{AdS}^{3}$
and satisfying the exponential growth condition \eqref{C}.
If $\varepsilon_{\Gamma}>0$,
then there exists a real number $m_{\Gamma}(k)$
$($given explicitly by $\eqref{Poincare_estimate})$
for $k\in \mathbb{N}$
such that
$\{(\operatorname{Re}(\psi_{m,3^{j}}))^{\Gamma}\}_{j=0}^{k-1}
\subset L^2_{\lambda_{m}}(\Gamma\backslash\mathrm{AdS}^{3})$
are linearly independent
for all integers $m > m_{\Gamma}(k)$.
\end{Proposition}

Postponing the proof of Proposition~\ref{mainthm} until the end of this section,
we prove Theorem~\ref{unbound}.

\begin{proof}[Proof of Theorem~\ref{unbound}]
We have an obvious equality of the multiplicity of $L^2$-eigenvalues,
$\mathcal{N}_{\Gamma\backslash\mathrm{AdS}^{3}} =
\mathcal{N}_{(g^{-1}\Gamma g)\backslash\mathrm{AdS}^{3}}$
for any $g\in G$
through the natural isomorphism
$\Gamma\backslash\mathrm{AdS}^{3} \cong
\big(g^{-1}\Gamma g\big)$ $\backslash\mathrm{AdS}^{3}$ as Lorentzian manifolds.
By replacing $\Gamma$ with $g^{-1}\Gamma g$ if necessary,
we may and do assume $\varepsilon_{\Gamma}>0$ by Proposition~\ref{strong}.
Then Proposition~\ref{mainthm} implies that
$L^2_{\lambda_{m}}\big(\Gamma\backslash \mathrm{AdS}^{3}\big)$ contains at least
$k$ linearly independent elements
if $m>m_{\Gamma}(k)$ for any fixed $k\in\mathbb{N}$,
which means
$\dim_{\mathbb{C}} L^2_{\lambda_{m}}(\Gamma\backslash\mathrm{AdS}^{3})
\geq k.$
Hence Theorem~\ref{unbound} follows.
\end{proof}

Kassel--Kobayashi~\cite{KaKob16} proved the non-vanishing of
the generalized Poincar\'e series $(\psi_{m,0})^{\Gamma}$
for sufficiently large $m\in\mathbb{N}$
by showing that the first term in the generalized Poincar\'e series
is larger at the origin than the sum of the remaining terms.
For this, they utilized the fact that $\psi_{m,0}(\grave{\ }E)=1$.
Our strategy for the proof of Proposition~\ref{mainthm} is along the same line,
however, there are some technical difficulties
since $\psi_{m,k}$ for $k\geq1$ vanishes at the origin. We~then make use of an observation that
$\psi_{m,k}$ decays more slowly at the origin than at infinity,
to be precise, by the following formula, see (\ref{polarFJ}):
\begin{gather*}
|\psi_{m,k}(x)|=\cosh^{-2m}(\|x\|/2) \tanh^{k}(\|x\|/2).
\end{gather*}
Actually, we use an analytic lemma (Lemma~\ref{koukou})
to prove that the first term in the generalized Poincar\'e series $(\psi_{m,k})^{\Gamma}$
is larger at points sufficiently close to the origin
than the sum of the remaining terms if $m\gg0$.
Moreover,
we use a combinatorial lemma (Lemma~\ref{sekoi})
to find points at which
leading terms of
$(\operatorname{Re}(\psi_{m,k}))^{\Gamma}$
do not cancel each other for any linear combination.

For $C,a,\varepsilon > 0$ and $s \in \mathbb{N}$, we set
\begin{align*}
m(C,a,\varepsilon,s) := \frac{(\log2)s + 2a\varepsilon + \log\big(1+2^{s}C{\rm e}^{6a\varepsilon}\big)}{\log\cosh\varepsilon}
\end{align*}
and
\begin{align*}
\tilde{m}(C,a,\delta,s) :=
\inf_{0<\varepsilon<\delta}m(C,a,\varepsilon,s).
\end{align*}
Note that $\tilde{m}(C,a,\delta,s)=O\big(\delta^{-2}\big)$ as $\delta\to 0$
and $=O(1)$ as $\delta\to \infty$.
\begin{Lemma}\label{koukou}
For any integer $m > m(C,a,\varepsilon,s)$
and any one-variable polynomial $f$ of degree $\leq s$ with non-negative coefficients,
\begin{gather*}
C\sum_{n=1}^{\infty}{\rm e}^{4a(n+1)\varepsilon}(\cosh 2n\varepsilon)^{-m}f(\tanh
2(n+1)\varepsilon) < (\cosh \varepsilon)^{-m}f(\tanh \varepsilon).
\end{gather*}
\end{Lemma}
\begin{proof}
We may assume that $f(x) = x^{j}$ for $j=0,1,\ldots,s$.
Since
\begin{gather*}
1\leq \frac{\tanh nx}{\tanh x} \leq n,\
(\cosh x)^{n} \leq \cosh nx
\end{gather*}
for $x\in\mathbb{R}$, we have
\begin{align*}
\text{(LHS)/(RHS)} &=
C\sum_{n=1}^{\infty}{\rm e}^{4a(n+1)\varepsilon}\bigg(\frac{\cosh 2n
\varepsilon}{\cosh \varepsilon}\bigg)^{-m}\bigg(\frac{\tanh
2(n+1)\varepsilon}{\tanh \varepsilon}\bigg)^j \nonumber
\\
&\leq C{\rm e}^{6a\varepsilon}\sum_{n=1}^{\infty}
\big({\rm e}^{2a\varepsilon}(\cosh \varepsilon)^{-m}\big)^{2n-1}(2(n+1))^s.
\end{align*}
We set $d:={\rm e}^{2a\varepsilon}(\cosh \varepsilon)^{-m}$.
Then $d<1$ by $m>m(C,a,\varepsilon,s)$. Since $n+1\leq 2^{n}$ for all $n\in\mathbb{N}$,
we have
\begin{gather*}
\text{(LHS)/(RHS)}\leq
2^{s}C{\rm e}^{6a\varepsilon}\sum_{n=1}^{\infty} (2^{s}d)^n
= 2^{s}C{\rm e}^{6a\varepsilon}\frac{2^{s}d}{1-2^{s}d}.
\end{gather*}
Again by $m > m(C,a,\varepsilon,s)$,
we have $2^{s}d<\big(1+2^{s}C{\rm e}^{6a\varepsilon}\big)^{-1}$.
Therefore we obtain
\begin{gather*}
\text{(LHS)/(RHS)}<1.\tag*{\qed}
\end{gather*}
\renewcommand{\qed}{}
\end{proof}
Let $\chi\colon\{\pm 1\} \rightarrow \{0,1\}$ be
the map defined by $\chi(1) = 0$ and $\chi(-1) = 1$.
For $a = (a_j)_{j=0}^{k-1} \in \{\pm1\}^{k}$ and an odd integer $N\geq 3$,
we set
\begin{align*}
\theta_{a,N} := \pi\sum_{i=0}^{k-1} (\chi(a_i) - \chi(a_{i-1}))N^{-i}.
\end{align*}
Here we use the convention $a_{-1}=1$.
\begin{Lemma}\label{sekoi}
For any $a=(a_{0},\ldots,a_{k-1})\in \{\pm1\}^{k}$ and any odd integer $N$,
we have
\begin{gather*}
a_j\cos\big(N^{j}\theta_{a,N}\big) > 0 \qquad \text{for}\quad j=0,1,\ldots,k-1.
\end{gather*}
\end{Lemma}

\begin{proof}
Since
$N^{k-1}\theta_{a,N} \equiv \pi \chi(a_{k-1}) \ (\mathrm{mod}\ 2\pi)$,
we have $\cos \big(N^{k-1} \theta_{a,N}\big) = a_{k-1}$.
It is easy to check that
$|N^{j} \theta_{a,N} - N^{j} \theta_{(a_0,\cdots,a_j),N}| < \pi/2$ for $j=0,1,\ldots,k-1$,
hence the signature of~$\cos (N^{j} \theta_{a,N})$ is equal to that of
$\cos (N^{j} \theta_{(a_0,\cdots,a_j),N}) = a_j$.
\end{proof}

\begin{Remark}
We have used the geometric progression $(N^{j})_{j=0}^{k-1}$ in Lemma~\ref{sekoi}.
On the other hand, an analogous statement does not hold if
we use arithmetic progressions. For example,
there does not exist $\theta \in \mathbb{R}$
satisfying $a_j\cos m_j\theta > 0$ for all $j=0,1,2,3,4$
if we choose $(a_j)_{j=0}^{4}=(1,1,1,-1,1)$ and
an arithmetic progression $(m_j)_{j=0}^{4}$.
\end{Remark}

For a discontinuous group $\Gamma$ and $k\in \mathbb{N}$,
one can take $m_{\Gamma}(k)$ in Proposition~\ref{Poincare_estimate} by
\begin{gather}
\label{Poincare_estimate}
m_{\Gamma}(k)=\inf_{(A,a)\in \mathrm{C_{exp}}(\Gamma)}
\max\big\{\tilde{m}\big(3^{k-1}A,a,\varepsilon_{\Gamma}/4,3^{k-1}\big)/2,a\big\},
\end{gather}
where $\mathrm{C_{exp}}(\Gamma):=\big\{(A,a)\in \mathbb{R}^{2} \mid
\forall x\in \mathrm{AdS}^{3},\forall R>0, N_{\Gamma}(x,R) < A{\rm e}^{aR}\big\}$.
Here,
we adopt the convention that $\inf_{\varnothing} f =\infty$ for a real-valued function $f$. In~particular,
$m_{\Gamma}(k)=\infty$ when $\mathrm{C_{exp}}(\Gamma)=\varnothing$
or $\varepsilon_{\Gamma}=0$.
\begin{proof}[Proof of Proposition~\ref{mainthm}]
By the exponential growth condition (\ref{C}), $\mathrm{C_{exp}}(\Gamma)\neq\varnothing$
and thus $m_{\Gamma}(k)<\infty$. We~take an integer $m >m_{\Gamma}(k)$.
Then there exist $\varepsilon$ with $0<\varepsilon<\varepsilon_{\Gamma}/4$
and $(A,a)\in\mathrm{C_{exp}}(\Gamma)$
satisfying the inequality
$m > \max\big\{m\big(3^{k-1}A,a,\varepsilon,3^{k-1}\big)/2,a\big\}$.

To see $\mathbb{C}$-linear independence of the real-valued functions
$\big\{(\operatorname{Re}(\psi_{m,3^{j}}))^{\Gamma}\big\}_{j=0}^{k-1}$,
it is enough to prove the non-vanishing of
the real part $\operatorname{Re}\big(\psi_{m,b}^{\Gamma}\big)
=(\operatorname{Re}(\psi_{m,b}))^{\Gamma}$ of
the generalized Poincar\'e series of a linear combination
\begin{gather*}
\psi_{m,b} := \sum_{j = 0}^{k-1} b_j\psi_{m,3^{j}}
\end{gather*}
for any
$b=(b_{0},b_{1},\ldots,b_{k-1}) \in \mathbb{R}^{k}\setminus\{0\}$.
By Lemma~\ref{inj-rad}, for $x\in B(4\varepsilon)$, we have
\begin{gather}
\label{separate}
\psi_{m,b}^{\Gamma}(\Gamma x)
= \psi_{m,b}(x) + \sum_{\substack{\gamma \in \Gamma \\
\|\gamma^{-1} x\| > 4\varepsilon}} \psi_{m,b}\big(\gamma^{-1} x\big).
\end{gather}
By (\ref{polarFJ}), for any $y \in \mathrm{AdS}^{3}$, we get
\begin{gather*}
|\psi_{m,b}(y)| \leq \bigg( \!\cosh \frac{\| y \|}{2} \bigg)^{-2m}
\sum_{j=0}^{k-1} |b_{j}|\bigg(\! \tanh \frac{\| y \|}{2} \bigg)^{3^{j}}.
\end{gather*}
We define
$a=(a_{j})_{j=0}^{k-1}$ by $a_{j}=1$ for $b_{j}\geq 0$ and
$a_{j}=-1$ for $b_{j}<0$,
and set
\begin{gather*}
f_{b}(u) := \sum_{j=0}^{k-1} b_{j}\cos\big(3^{j}\theta_{a,3}\big)u^{3^{j}}.
\end{gather*}
We note that all the coefficients of $f_{b}$ are non-negative by Lemma~\ref{sekoi}.
Moreover, we get
$\big| \cos \big(3^{j}\theta_{a,3}\big) \big|^{-1}
\leq 3^{k-1}$ for all $j=0,1,\ldots,k-1$
by using the inequality $\sin(\pi x/2)\geq x$ for~$0\leq x\leq1$.
Thus
\begin{gather*}
|\psi_{m,b}(y)| \leq 3^{k-1}\bigg(\!\cosh \frac{\| y \|}{2} \bigg)^{-2m}
f_{b}\bigg(\! \tanh \frac{\| y \|}{2} \bigg)
\end{gather*}
and, for any $x\in B(4\varepsilon)$, we have
\begin{align}
\bigg|\!\!\!\sum_{\substack{\gamma \in \Gamma \\ \|\gamma^{-1} x\| > 4\varepsilon}} \operatorname{Re}\big(\psi_{m,b}\big(\gamma^{-1} x\big)\big)\bigg|
&\leq \sum_{n = 1}^{\infty} \sum_{\substack{\gamma \in \Gamma \\
4\varepsilon n < \|\gamma^{-1} x\| \leq 4\varepsilon (n + 1) }}
|\psi_{m,b}(\gamma^{-1} x)| \nonumber \\
&\leq 3^{k-1}\sum_{n = 1}^{\infty} N_{\Gamma}(x, 4\varepsilon (n + 1))\left(\cosh
2\varepsilon n \right)^{-2m} f_{b}\left(\tanh 2\varepsilon (n+1)\right)
\nonumber \\
&\leq 3^{k-1}A\sum_{n = 1}^{\infty} {\rm e}^{4a\varepsilon (n+1)}\left(\cosh 2\varepsilon n
\right)^{-2m} f_{b}\left(\tanh 2\varepsilon (n+1)\right)
\nonumber \\
&< \left(\cosh \varepsilon\right)^{-2m} f_{b}
\left(\tanh \varepsilon\right).
\label{small}
\end{align}
The third and forth inequalities respectively follow from
the exponential growth condition (\ref{C}) and Lemma~\ref{koukou}.
On the other hand, we set
\begin{gather*}
x_{a,\varepsilon}:=
k\bigg(\frac{\theta_{a,3}}{2}\bigg)a(\varepsilon)k\bigg(\frac{\theta_{a,3}}{2}\bigg)^{-1} \in B(4\varepsilon).
\end{gather*}
Then it follows from (\ref{polarFJ}) that
\begin{align}
\label{large}
\operatorname{Re} \psi_{m,b}(x_{a,\varepsilon})
=\left(\cosh \varepsilon\right)^{-2m} f_{b}
\left(\tanh \varepsilon\right).
\end{align}
By (\ref{separate}), (\ref{small}), and (\ref{large}),
we obtain $(\operatorname{Re}(\psi_{m,b}))^{\Gamma}(\Gamma x_{a,\varepsilon})\neq 0$.
Hence we complete the proof
by the continuity of $\psi_{m,b}^{\Gamma}$ (Fact~\ref{Poincare}).
\end{proof}
\section{Proof of Theorem~\ref{mthm2}}
\label{proof-thm2-section}
In this section, we prove Theorem~\ref{mthm2}
by applying Proposition~\ref{mainthm}. We~work in the following setting. We~allow $\Delta$ to have torsion.
\begin{Setting}\quad
\label{setting}
\begin{itemize}\itemsep=0pt
\item
$\Delta$ is a discrete subgroup of $\grave{\ }G = \mathrm{PSL}(2,\mathbb{R})$.
\item
$j, \rho\colon\Delta\rightarrow\grave{\ }G$ are two group homomorphisms
with $j$ injective and discrete.
\item
$\Delta^{j,\rho}$ is a discrete subgroup of $G=\grave{\ }G\times\grave{\ }G$
given by $\{(j(\gamma),\rho(\gamma)) \mid \gamma \in \Delta\}$.
\end{itemize}
\end{Setting}

We use the following structural results of discontinuous groups
for the proof of Theorem~\ref{mthm2}.
\begin{Fact}[{\cite[Lemma~9.2]{KaKob16}}]
\label{explicit}
Let $\Gamma$ be a finitely generated discrete subgroup of $G$
acting properly discontinuously on $\mathrm{AdS}^{3}$.
Then $\Gamma$ is of either type $(i)$ or $(ii)$ as follows:
\begin{enumerate}\itemsep=0pt
\setlength{\leftskip}{1.1cm}
\item[type\phantom{$i$}~$(i)$]
$\Gamma$ is of the form $\Delta^{j,\rho}$ up to switching the two factors,
\item[type~$(ii)$]
$\Gamma$ is contained in a conjugate of
$\grave{\ }G \times\grave{\ }K$ or
$\grave{\ }K\times\grave{\ }G$.
\end{enumerate}
\end{Fact}

A non-elementary discrete subgroup $\Gamma$ of
a connected linear real reductive Lie group $L$ of real rank $1$
is called \textit{convex cocompact}
if $\Gamma$ acts cocompactly on the convex hull of
its limit set in the Riemannian symmetric space associated to $L$.
For example, cocompact lattices and Schottky groups are convex cocompact.
More generally, one may think of the notion of convex cocompactness of
discontinuous groups for $\mathrm{AdS}^{3}$:
\begin{Definition}[{\cite[Definition~9.1]{KaKob16}}]
\label{convex_cocompact}
A discontinuous group $\Gamma$ for $\mathrm{AdS}^{3}$ is called convex cocom\-pact
if $\Gamma$ is of the form $\Delta^{j,\rho}$
up to finite index and switching the two factors,
where $\Delta$ is torsion-free and
$j(\Delta)$ is convex cocompact in $\grave{\ }G$.
\end{Definition}

We note that
a discontinuous group $\Delta^{j,\rho}$ acts cocompactly on $\mathrm{AdS}^{3}$
if and only if
$j(\Delta)$ is cocompact in $\grave{\ }G$
because $\Delta^{j,\rho}$ is isomorphic to $j(\Delta)$ as abstract groups.
By Fact~\ref{explicit}, discontinuous groups acting cocompactly on $\mathrm{AdS}^{3}$
are convex cocompact.

\subsection[Proof of Theorem~1.4 for Gamma of type (i)]
{Proof of Theorem~\ref{mthm2} for $\boldsymbol\Gamma$ of type ($\boldsymbol i$)}

In this subsection, we prove Theorem~\ref{mthm2} for $\Gamma$ of type $(i)$.
For this, we use the constant $C_{\rm Lip}(j,\rho)$
introduced by Kassel~\cite{Ka09} and Gu\'eritaud--Kassel~\cite{GuKa17},
which quantifies the properness of the action
of $\Delta^{j,\rho}$ on $\mathrm{AdS}^{3}$.

\begin{Definition}
\label{def-lip}
Let $d_{\mathbb{H}^{2}}$ be the hyperbolic distance of
the $2$-dimensional hyperbolic space \mbox{$\mathbb{H}^{2}(\cong\grave{\ }G/\grave{\ }K)$}. In~Setting~\ref{setting},
we denote by $C_{\rm Lip}(j,\rho)$ the infimum of Lipschitz constants
\begin{gather*}
\operatorname{Lip}(f)=\sup_{y\neq y'}\frac{d_{\mathbb{H}^{2}}(f(y),f(y'))}{d_{\mathbb{H}^{2}}(y,y')}
\end{gather*}
of maps $f\colon \mathbb{H}^{2} \rightarrow \mathbb{H}^{2}$
that are $(j,\rho)$-equivariant.
\end{Definition}
The map $(j,\rho) \mapsto C_{\rm Lip}(j,\rho)$
is continuous over the set of
$(j,\rho)\in\operatorname{Hom}(\Delta,\grave{\ }G)^{2}$
such that $j$ is injective and $j(\Delta)$ is convex cocompact in $\grave{\ }G$
\cite[Proposition~1.5]{GuKa17}.
\begin{Fact}[\cite{GuKa17,Ka09}]
\label{proper}
Assume that $\Delta$ is finitely generated.
Then the action of $\Delta^{j,\rho}$ on $\mathrm{AdS}^{3}$ is pro\-perly discontinuous
if and only if $\operatorname{min}\big\{C_{\rm Lip}(j,\rho),C_{\rm Lip}(\rho,j)\big\} < 1$.
\end{Fact}

\begin{Remark}
In the setting of Fact~\ref{proper},
if $C_{\rm Lip}(\rho,j)<1$, then $\rho$ is injective and discrete.
Moreover, if $j(\Delta)$ is convex cocompact, then so is $\rho(\Delta)$.
\end{Remark}
Therefore,
Theorem~\ref{mthm2} for $\Gamma$ of type $(i)$
reduces to the following:
\begin{Theorem}
\label{non-standard}
In Setting~$\ref{setting}$, we assume that
$\Delta$ is finitely generated and that $C_{\rm Lip}(j,\rho) < 1$.
Then there exists a constant $\mu_{1}>0$ independent of $j,\rho$ and $\Delta$ such that
for any $m,k\in\mathbb{N}$ with $m > 3^{k}\mu_{1}(1-C_{\rm Lip}(j,\rho))^{-2}$,
\begin{gather*}
\mathcal{N}_{\Delta^{j,\rho}\backslash\mathrm{AdS}^{3}}(\lambda_{m}) \geq k.
\end{gather*}
\end{Theorem}

For the proof of Theorem~\ref{non-standard}, we need two results
from Kassel--Kobayashi~\cite{KaKob16} applied to our setting
$G=\grave{\ }G\times\grave{\ }G$.
If a discontinuous group
$\Gamma$ satisfies the assumption of Fact~\ref{alphacounting} below,
then it is $((1-\alpha)/2,0)$-sharp
in the sense of \cite[Definition~4.2]{KaKob16}.
Hence we get the following by~applying \cite[Lemma~4.6.4]{KaKob16}:
\begin{Fact}[\cite{KaKob16}]
\label{alphacounting}
Let $\Gamma \subset G$ be a discontinuous group for $\mathrm{AdS}^{3}$. We~assume that there exists $0\leq\alpha<1$ such that
$\|\gamma_2\|\leq\alpha\|\gamma_1\|$ or $\|\gamma_1\|\leq\alpha\|\gamma_2\|$
for any $(\gamma_1,\gamma_2)\in\Gamma$.
Then there exists $c>0$ independent of $\alpha$ and $\Gamma$
such that for any $x\in\mathrm{AdS}^3$ and any $R>0$,
\begin{gather*}
N_{\Gamma}(x,R)\leq \#(\Gamma\cap K)c{\rm e}^{8R(1-\alpha)^{-1}}.
\end{gather*}
\end{Fact}

The following theorem
traces back to the Kazhdan--Margulis theorem for discrete subgroups of
semisimple groups.
\begin{Fact}[{\cite[Proposition~8.14]{KaKob16}}]
\label{Kazhdan-Margulis}
There exists a constant $r>0$ satisfying the following property:
for any discrete subgroup $\grave{\ }\Gamma$ of $\grave{\ }G$,
there exists $\grave{\ }g \in \grave{\ }G$ such that
$\|\grave{\ }\gamma\|\geq r$ for all
$\grave{\ }\gamma\in
\grave{\ }g^{-1}\grave{\ }\Gamma \grave{\ }g\setminus\{\grave{\ }E\}$.
\end{Fact}

In the following, we use the upper half plane model
$\big\{z=x+\sqrt{-1}y\in\mathbb{C}\mid \operatorname{Im}z>0\big\}$
equipped with the metric tensor ${\rm d}s^2=\big({\rm d}x^2+{\rm d}y^2\big)/y^2$
for the hyperbolic space~$\mathbb{H}^2$.
Then $\|\grave{\ }g\|$ is equal to the hyperbolic distance
$d_{\mathbb{H}^{2}}\big(\grave{\ }g\sqrt{-1},\sqrt{-1}\big)$
for $\grave{\ }g\in\mathrm{AdS}^{3}\cong\grave{\ }G$ (see, e.g., \cite[equation~(A.1)]{GuKa17}).

\begin{proof}[Proof of Theorem~\ref{non-standard}]
The idea of the proof is similar to \cite[Theorem~9.9]{KaKob16},
however, we give a proof for the sake of completeness.
By Fact~\ref{Kazhdan-Margulis},
replacing $j$ by some conjugate under $\grave{\ }G$,
we may assume $\|j(\gamma)\|\geq r$ for any
$\gamma\in\Delta\setminus\{\grave{\ }E\}$. In~particular, $\Gamma\cap K=\{E\}$ for such $j$ and for any $\rho$. We~fix $\delta>0$ such that
\begin{gather*}
\alpha := C_{\rm Lip}(j,\rho) + \delta<1.
\end{gather*}
Then, replacing $\rho$ by some conjugate under $\grave{\ }G$,
we may assume
\begin{gather}
\label{rho-est}
\|\rho(\gamma)\| \leq \alpha \|j(\gamma)\|\qquad
\text{for any} \quad\gamma\in\Delta.
\end{gather}
Indeed, by Definition~\ref{def-lip},
there exists a $(j,\rho)$-equivariant map
$f_{\delta}\colon\mathbb{H}^2\rightarrow\mathbb{H}^2$
satisfying $\operatorname{Lip}(f_{\delta})$ $< \alpha$. We~take $g_{\delta}\in\grave{\ }G$ such that $g_{\delta}\sqrt{-1}=f_{\delta}\big(\sqrt{-1}\big)$.
Then, for any $\gamma\in\Delta$, we have
\begin{gather*}
\big\|g_{\delta}^{-1}\rho(\gamma)g_{\delta}\big\| =
d_{\mathbb{H}^{2}}\big(f_{\delta}(\sqrt{-1}), \rho(\gamma)f_{\delta}\big(\sqrt{-1}\big)\big)
< \alpha d_{\mathbb{H}^{2}}\big(\sqrt{-1},j(\gamma)\sqrt{-1}\big)
=\alpha \|j(\gamma)\|.
\end{gather*}
Hence (\ref{rho-est}) holds by replacing $\rho$ with $g_{\delta}^{-1}\rho(\cdot)g_{\delta}$,
and therefore we get
\begin{gather*}
N_{\Gamma}(x,R) \leq c{\rm e}^{8R(1-(C_{\rm Lip}(j,\rho)+\delta))^{-1}}
\end{gather*}
by Fact~\ref{alphacounting}. Then the constant $\varepsilon_{\Gamma}$
in (\ref{varepsilon}) has the following lower bound:
\begin{gather*}
3\varepsilon_{\Gamma}
=\inf_{\gamma\in\Delta\setminus\{\grave{\ }E\}}
\left|\|j(\gamma)\| - \|\rho(\gamma)\|\right|
\geq \inf_{\gamma\in\Delta\setminus\{\grave{\ }E\}}(1-\alpha)
\|j(\gamma)\| \geq r(1-\alpha).
\end{gather*}
Note that
$\log\cosh t=O\big(t^{2}\big)$ as $t\to0$.
By the explicit description (\ref{Poincare_estimate})
of $m_{\Gamma}(k)$,
Theorem~\ref{non-standard} follows from Proposition~\ref{mainthm}.
\end{proof}

\subsection[Proof of Theorem~1.4 for Gamma of type (ii)]
{Proof of Theorem~\ref{mthm2} for $\boldsymbol\Gamma$ of type ($\boldsymbol {ii}$)}

In this subsection, we prove Theorem~\ref{mthm2} for the case
where $\Gamma$ is standard.
For this, we use the following fact by Kobayashi~\cite{Kob98} and Kassel~\cite{Ka12}
applied to our $\mathrm{AdS}^{3}$ setting, which gives the stability for properness
under any small deformation of standard convex cocompact discontinuous groups.
\begin{Fact}[{\cite[Theorem~1.4]{Ka12}}]
\label{deform}
Let $\Gamma$ be a convex cocompact discrete subgroup of $\grave{\ }G \times \grave{\ }K$.
Then for any $\alpha,\beta>0$,
there exists a neighborhood $W\subset\operatorname{Hom}(\Gamma,G)$
of the natural inclusion $\Gamma \subset G$ such that
for any $\varphi\in W$,
\begin{align*}
\left|\mu(\varphi(\gamma))-\mu(\gamma)\right|\leq
\begin{cases}
\alpha\left|\mu(\gamma)\right|&\text{if}\quad\gamma \in \Gamma\setminus K,\\
\beta&\text{if}\quad\gamma \in \Gamma\cap K,
\end{cases}
\end{align*}
where $\mu(g_{1},g_{2}):=(\|g_{1}\|,\|g_{2}\|)\in\mathbb{R}^{2}$ for $(g_{1},g_{2})\in G$,
$\|\cdot\|$ is given in Definition~$\ref{def-pball}$,
and $\left|(x_1,x_2)\right|:=\sqrt{x_{1}^{2} + x_{2}^{2}}$ for $(x_{1},x_{2})\in\mathbb{R}^2$.
\end{Fact}

We introduce the following terminology for the estimate of the discrete spectrum
since a~discontinuous group $\Gamma$ is not necessarily torsion-free.
Let $\operatorname{pr}_{j}\colon G=\grave{\ }G\times\grave{\ }G
\rightarrow \grave{\ }G$
be the $j$-th projection ($j=1,2$). In~the following definition, we assume that $\operatorname{pr}_{2}(\Gamma)$ is bounded.
Then the group $\Gamma_1:=\mathrm{ker}(\operatorname{pr}_{1}|_{\Gamma})$ is cyclic
since $\Gamma_1$ is a discrete subgroup of a conjugate of the product group
$\{\grave{\ }E\}\times\grave{\ }K$ $(\cong \mathbb{R}/\mathbb{Z})$.

\begin{Definition}
\label{classn}
A discrete subgroup $\Gamma$ of $G$ is said to be standard of class $n$
if $\operatorname{pr}_{2}(\Gamma)$ is bounded and
the cyclic group $\Gamma_1=\mathrm{ker}(\operatorname{pr}_{1}|_{\Gamma})$
is of order $n$.
\end{Definition}

{\samepage\begin{Remark}\quad
\begin{enumerate}\itemsep=0pt
\item[(1)]
If $\Gamma$ is torsion-free, then it is of class $1$.
\item[(2)]
If $\operatorname{pr}_{2}(\Gamma)$ is bounded
for a discrete subgroup $\Gamma$ of $G$, then
the group $\operatorname{pr}_{1}(\Gamma)$ is discrete
in~$\grave{\ }G$.
Moreover, if $\Gamma$ is of class $1$, then it is of the form $\Delta^{j,\rho}$
such that $\Delta = \operatorname{pr}_{1}(\Gamma)$ and~$C_{\rm Lip}(j,\rho)=0$.
\end{enumerate}
\end{Remark}

}

Let $r>0$ be the constant in Fact~\ref{Kazhdan-Margulis}.
For an integer $n \geq 2$,
we define a positive num\-ber~$\eta_{n}$~by
\begin{gather*}
\cosh \eta_{n} :=1+2\bigg(\sinh \frac{r}{4} \sin \frac{\pi}{n}\bigg)^{2}.
\end{gather*}
We get the following by easy computations:
\begin{Lemma}
\label{computation}
By an abuse of notation, we regard $k(\theta)$, $a(t)$ in \eqref{k-a}
as elements of $\grave{\ }G=\mathrm{PSL}(2,\mathbb{R})$. Then
\begin{gather*}
\bigg\|a\bigg(\frac{r}{8}\bigg)^{-1}k\bigg(\frac{j\pi}{n}\bigg)a\bigg(\frac{r}{8}\bigg)\bigg\| \geq \eta_{n}\qquad
\text{for}\quad j=1,\ldots,n-1.
\end{gather*}
\end{Lemma}

We give a uniform estimate of $\varepsilon_{\Gamma}$ in (\ref{varepsilon})
and $N_{\Gamma}(x,R)$ in (\ref{orbit-count})
for standard discrete subgroups $\Gamma$ of class $n$
after taking a conjugation of $\Gamma$.

\begin{Lemma}
\label{n-standard}
Let $\Gamma$ be a standard discrete subgroup of class $n \geq 2$.
There exists $g\in G$ such that
$\varepsilon_{g^{-1}\Gamma g}\geq\operatorname{min}\{\eta_n/3,r/6\}$
and $N_{g^{-1}\Gamma g}(x,R)<c{\rm e}^{16R}$ for any $x\in \mathrm{AdS}^{3}$ and any $R>0$.
\end{Lemma}

\begin{proof}
Let $\Gamma_1=\mathrm{ker}(\operatorname{pr}_{1}|_{\Gamma})$
as in Definition~\ref{classn}.
Since $\Gamma$ is of class $n$,
the group $\operatorname{pr}_{2}(\Gamma_{1})$ is generated by
$k(\pi/n)\in \grave{\ }G=\mathrm{PSL}(2,\mathbb{R})$. We~take $\grave{\ }g \in \grave{\ }G$ in Fact~\ref{Kazhdan-Margulis} applied
to $\grave{\ }\Gamma=\operatorname{pr}_{1}(\Gamma)$ and
set $g:=(\grave{\ }g,a(r/8))\in G$.
Replacing $\Gamma$ by $g^{-1}\Gamma g$, we get
$\|\gamma_{1}\|\geq r$ for
$(\gamma_{1},\gamma_{2})\in\Gamma\setminus\Gamma_{1}$
by Fact~\ref{Kazhdan-Margulis} and
$\|\gamma_{2}\|\geq \eta_n$ for $(\gamma_1,\gamma_2)\in\Gamma_1\setminus\{E\}$
by Lemma~\ref{computation}.
Moreover, if $(\gamma_{1},\gamma_{2})\in\Gamma$, then
$\|\gamma_{2}\|=\|a(r/8)^{-1}ka(r/8)\|$ for some $k\in\grave{\ }K$,
hence $\|\gamma_{2}\|\leq r/2$ because
$\|g_{1}g_{2}\|\leq\|g_{1}\|+\|g_{2}\|$ for $g_{1},g_{2} \in \grave{\ }G$
and since
$\|a(t)\|=2t$ for $t\geq 0$ and $\|k\|=0$ for $k\in\grave{\ }K$.
To summarize,
\begin{gather*}
\begin{cases}
\|\gamma_{2}\|\leq \frac{r}{2}\leq\frac{\|\gamma_{1}\|}{2} &
\text{if}\quad (\gamma_{1},\gamma_{2})\in\Gamma\setminus\Gamma_{1},
\\[1ex]
\|\gamma_{2}\|\geq \eta_n &
\text{if}\quad(\gamma_{1},\gamma_{2})\in\Gamma_{1}\setminus\{E\}.
\end{cases}
\end{gather*}
Then $\varepsilon_{\Gamma}\geq\operatorname{min}\{\eta_n/3,r/6\}$
and $\Gamma\cap K=\{E\}$.
Moreover, $\|\gamma_{1}\|\leq\|\gamma_{2}\|/2$ or
$\|\gamma_{2}\|\leq\|\gamma_{1}\|/2$ for any $(\gamma_{1},\gamma_{2})\in\Gamma$
and thus $N_{\Gamma}(x,R)<c{\rm e}^{16R}$
for~any~$x\in \mathrm{AdS}^{3}$ and any $R>0$
by Fact~\ref{alphacounting}.
\end{proof}

\begin{Theorem}
\label{standard}
There exists a constant $\mu_{n}>0$ depending only on $n$ such that
for any convex cocompact standard discrete subgroup $\Gamma$ of class $n$
and any $m,k\in\mathbb{N}$ with $m>3^k\mu_{n}$,
\begin{gather*}
\widetilde{\mathcal{N}}_{\Gamma\backslash\mathrm{AdS}^{3}}(\lambda_{m})\geq k.
\end{gather*}
\end{Theorem}

\begin{proof}
{\samepage If $n=1$, then this follows from Theorem~\ref{non-standard}
since convex cocompact discontinuous groups are finitely generated,
hence we assume that $n\geq 2$. In~this case, we shall prove that $\Gamma$ and its small deformation
are standard of class $n$. When $n\geq2$,
the group $\Gamma_{1}=\mathrm{ker}(\operatorname{pr}_{1}|_{\Gamma})$
is a~cyclic group of order $n$.
By Fact~\ref{Kazhdan-Margulis},
replacing $\Gamma$ by some conjugate under $\grave{\ }G\times\{\grave{\ }E\}$,
we may and do assume $\|\gamma_1\|\geq r$
for any $(\gamma_1,\gamma_2)\in\Gamma\setminus\Gamma_1$.
By Fact~\ref{deform}, there exists a neighborhood~$W$ of~the natural inclusion
$\Gamma \subset G$ such that for any $\varphi\in W$,
the restriction of $\varphi$ to the finite subgroup~$\Gamma_{1}$ is injective and the inequalities
\begin{align}
\label{estimate}
\begin{cases}
\|\varphi_{1}(\gamma)\| \geq \frac{1}{2}r,\quad
\|\varphi_{2}(\gamma)\|\leq \frac{1}{2}\|\varphi_{1}(\gamma)\| &
\text{if}\quad\gamma \in \Gamma\setminus\Gamma_1,
\\[1ex]
\left|\mu(\varphi(\gamma))\right| < \frac{1}{2}r &
\text{if}\quad\gamma \in \Gamma_1
\end{cases}
\end{align}
hold where $\varphi_{i}=\operatorname{pr}_{i}\circ\varphi$ for $i=1,2$.
Then $\varphi$ is injective and discrete.

}

We claim $\varphi_1(\Gamma_1)$ is trivial.
Indeed, if there exists $\gamma \in \Gamma_1\setminus\{E\}$
such that $\varphi_1(\gamma) \neq \grave{\ }E$, then the normalizer of
$\varphi(\Gamma_1)$ in $G$ is contained in $\grave{\ }K_1\times \grave{\ }G$, where
$\grave{\ }K_1$ is the maximal compact subgroup of $\grave{\ }G$
containing $\varphi_1(\Gamma_1)$.
Hence $\varphi(\Gamma) \subset \grave{\ }K_1\times \grave{\ }G$.
By the inequalities (\ref{estimate}),
$\varphi(\Gamma)$ is finite, hence~$\Gamma$ is also finite.
This contradicts the assumption that $\Gamma$ is non-elementary.
Thus $\varphi_1(\Gamma_1)$ is trivial and
$\varphi_2(\Gamma_1)$ is non-trivial.
Hence the normalizer of $\varphi(\Gamma_{1})$ in $G$ is contained in
$\grave{\ }G\times \grave{\ }K_{2}$,
where~$\grave{\ }K_2$ is the maximal compact subgroup of $\grave{\ }G$
containing $\varphi_2(\Gamma_1)$.
Therefore $\operatorname{pr}_{2}(\varphi(\Gamma))$ is bounded.
Moreover
$\varphi(\Gamma)_{1}=\varphi(\Gamma_1)$ by the inequalities (\ref{estimate}),
hence the discrete subgroup $\varphi(\Gamma)$ is standard of class $n$.
By the explicit description (\ref{Poincare_estimate}) of $m_{\Gamma}(k)$
and Lemma~\ref{n-standard},
Theorem~\ref{standard} follows
from Proposition~\ref{mainthm}.
\end{proof}

\begin{Remark}
In the above proof, we have shown that a convex cocompact standard discrete subgroup
$\Gamma$ of class $n\geq 2$ and its small deformation are standard of class $n$.
Therefore we obtain a stronger result that
\begin{align*}
\widetilde{\mathcal{N}}_{\Gamma\backslash\mathrm{AdS}^{3}}(\lambda_{m})=\infty
\end{align*}
for any convex cocompact standard discrete subgroup
$\Gamma$ of class $n\geq 2$ and any integer $m>3\mu_{n}$
if the following statement holds:
$\mathcal{N}_{\Gamma\backslash\mathrm{AdS}^{3}}(\lambda_{m})=\infty$
for any standard discrete subgroup $\Gamma$ and any $m\in\mathbb{N}$ such that
$\mathcal{N}_{\Gamma\backslash\mathrm{AdS}^{3}}(\lambda_{m})\geq1$.
The latter statement is discussed in~\cite{AnaKaKob} by using discretely
decomposable blanching laws of unitary representations (cf.~\cite{KaKob19}).
\end{Remark}

Thus the proof of Theorem~\ref{mthm2} is completed.

\subsection*{Acknowledgements}
The author would like to express his sincere gratitude to
Professor Toshiyuki Kobayashi whose suggestions
led him to study the multiplicities of $L^2$-eigenvalues
for anti-de Sitter manifolds.
He~also would like to show his appreciation to Dr.\ Hiroyoshi Tamori
whose comments led him to an explicit description of $m(C,a,\varepsilon, s)$
in Lemma~\ref{koukou}.
Thanks are also due to the anonymous referees
for their helpful comments to improve the paper.
This work was supported by JSPS KAKENHI Grant Number 18J20157
and the Program for Leading Graduate Schools, MEXT, Japan.

\pdfbookmark[1]{References}{ref}
\LastPageEnding


\begin{thebibliography}{99}
\footnotesize\itemsep=0pt

\bibitem{Ben96}
Benoist Y., Actions propres sur les espaces homog\`enes r\'eductifs,
 \href{https://doi.org/10.2307/2118594}{\textit{Ann. of Math.}} \textbf{144} (1996), 315--347.

\bibitem{Bourbaki-integral}
Bourbaki N., Integration. {II}. {C}hapters 7--9, Elements of Mathematics
 (Berlin), \href{https://doi.org/10.1007/978-3-662-07931-7}{Springer-Verlag}, Berlin, 2004.

\bibitem{CaMa62}
Calabi E., Markus L., Relativistic space forms, \href{https://doi.org/10.2307/1970419}{\textit{Ann. of Math.}}
 \textbf{75} (1962), 63--76.

\bibitem{FoSt18}
Fox J., Strichartz R.S., Unexpected spectral asymptotics for wave equations on
 certain compact spacetimes, \href{https://doi.org/10.1007/s11854-018-0059-2}{\textit{J.~Anal. Math.}} \textbf{136} (2018),
 209--251, \href{https://arxiv.org/abs/1407.2517}{arXiv:1407.2517}.

\bibitem{Go85}
Goldman W.M., Nonstandard {L}orentz space forms, \href{https://doi.org/10.4310/jdg/1214439567}{\textit{J.~Differential Geom.}}
 \textbf{21} (1985), 301--308.

\bibitem{GuKa17}
Gu\'eritaud F., Kassel F., Maximally stretched laminations on geometrically
 finite hyperbolic manifolds, \href{https://doi.org/10.2140/gt.2017.21.693}{\textit{Geom. Topol.}} \textbf{21} (2017),
 693--840, \href{https://arxiv.org/abs/1307.0250}{arXiv:1307.0250}.

\bibitem{Kan19}
Kannaka K., Counting orbits of certain infinitely generated non-sharp
 discontinuous groups for the anti-de {S}itter space, \href{https://arxiv.org/abs/1907.09303}{arXiv:1907.09303}.

\bibitem{Ka09}
Kassel F., Quotients compacts d'espaces homog\`enes r\'eels ou $p$-adiques,
 Ph.D.~Thesis, {U}niversit\'e Paris-Sud, 2009.

\bibitem{Ka12}
Kassel F., Deformation of proper actions on reductive homogeneous spaces,
 \href{https://doi.org/10.1007/s00208-011-0672-1}{\textit{Math. Ann.}} \textbf{353} (2012), 599--632, \href{https://arxiv.org/abs/0911.4247}{arXiv:0911.4247}.

\bibitem{KaKob16}
Kassel F., Kobayashi T., Poincar\'e series for non-{R}iemannian locally
 symmetric spaces, \href{https://doi.org/10.1016/j.aim.2015.08.029}{\textit{Adv. Math.}} \textbf{287} (2016), 123--236,
 \href{https://arxiv.org/abs/1209.4075}{arXiv:1209.4075}.

\bibitem{KaKob19}
Kassel F., Kobayashi T., Spectral analysis on standard locally homogeneous
 spaces, \href{https://arxiv.org/abs/1912.12601}{arXiv:1912.12601}.

\bibitem{KaKob20PJA}
Kassel F., Kobayashi T., Spectral analysis on pseudo-{R}iemannian locally
 symmetric spaces, \href{https://doi.org/10.3792/pjaa.96.013}{\textit{Proc. Japan Acad. Ser.~A Math. Sci.}} \textbf{96}
 (2020), 69--74, \href{https://arxiv.org/abs/2001.03292}{arXiv:2001.03292}.

\bibitem{AnaKaKob}
Kassel F., Kobayashi T., Analyticity of {P}oincar\'e series on standard
 non-{R}iemannian locally symmetric spaces, in preparation.

\bibitem{Kl96}
Klingler B., Compl\'etude des vari\'et\'es lorentziennes \`a courbure
 constante, \href{https://doi.org/10.1007/BF01445255}{\textit{Math. Ann.}} \textbf{306} (1996), 353--370.

\bibitem{Kob89}
Kobayashi T., Proper action on a homogeneous space of reductive type,
 \href{https://doi.org/10.1007/BF01443517}{\textit{Math. Ann.}} \textbf{285} (1989), 249--263.

\bibitem{Kob96}
Kobayashi T., Criterion for proper actions on homogeneous spaces of reductive
 groups, \textit{J.~Lie Theory} \textbf{6} (1996), 147--163.

\bibitem{Kob98}
Kobayashi T., Deformation of compact {C}lifford--{K}lein forms of
 indefinite-{R}iemannian homogeneous manifolds, \href{https://doi.org/10.1007/s002080050153}{\textit{Math. Ann.}}
 \textbf{310} (1998), 395--409.

\bibitem{Ko01}
Kobayashi T., Discontinuous groups for non-{R}iemannian homogeneous spaces, in
 Mathematics Unlimited~-- 2001 and Beyond, \href{https://doi.org/10.1007/978-3-642-56478-9_37}{Springer}, Berlin, 2001, 723--747.

\bibitem{Kob16}
Kobayashi T., Intrinsic sound of anti-de {S}itter manifolds, in Lie Theory and
 its Applications in Physics, \textit{Springer Proc. Math. Stat.}, Vol.~191,
 \href{https://doi.org/10.1007/978-981-10-2636-2_6}{Springer}, Singapore, 2016, 83--99, \href{https://arxiv.org/abs/1609.05986}{arXiv:1609.05986}.

\bibitem{KoNa06}
Kobayashi T., Nasrin S., Deformation of properly discontinuous actions of
 {${\mathbb Z}^k$} on {${\mathbb R}^{k+1}$}, \href{https://doi.org/10.1142/S0129167X06003862}{\textit{Internat.~J. Math.}}
 \textbf{17} (2006), 1175--1193, \href{https://arxiv.org/abs/math.DG/0603318}{arXiv:math.DG/0603318}.

\bibitem{KuRa85}
Kulkarni R.S., Raymond F., {$3$}-dimensional {L}orentz space-forms and
 {S}eifert fiber spaces, \href{https://doi.org/10.4310/jdg/1214439564}{\textit{J.~Differential Geom.}} \textbf{21} (1985),
 231--268.

\bibitem{functional-analysis}
Rudin W., Functional analysis, 2nd ed., \textit{International Series in Pure and
 Applied Mathematics}, McGraw-Hill, Inc., New York, 1991.

\bibitem{Sa00}
Salein F., Vari\'et\'es anti-de {S}itter de dimension 3 exotiques, \href{https://doi.org/10.5802/aif.1754}{\textit{Ann.
 Inst. Fourier (Grenoble)}} \textbf{50} (2000), 257--284.

\end{thebibliography}
\end{document}